\documentclass[11pt,a4paper]{article}
\usepackage[top=2.5cm,bottom=2.5cm,left=2.2cm,right=2.2cm]{geometry}
\usepackage[T1]{fontenc}
\usepackage[utf8]{inputenc}
\usepackage{color}
\usepackage{graphicx}
\usepackage{amsfonts}
\usepackage{extarrows}
\usepackage{amsmath,amsthm,amssymb,color}
\usepackage{hyperref}
\usepackage{eepic}
\usepackage{lineno}
\usepackage{enumerate}	
\usepackage{paralist}
\usepackage{cite}
\usepackage{algorithm}
\usepackage{algorithmicx}
\usepackage{algpseudocode}
\usepackage{authblk}
\usepackage{mathtools}
\usepackage{pifont}
\usepackage[numbers,sort&compress]{natbib}

\usepackage{tikz}

\hypersetup
{
	colorlinks=true,
	linkcolor=blue,
	filecolor=blue,
	urlcolor=blue,
	citecolor=cyan,
}

\newtheorem{theorem}{Theorem}[section]

\newtheorem{corollary}[theorem]{Corollary}

\theoremstyle{definition}

\newtheorem{claim}{\indent Claim}

\newtheorem{case}{Case}[section]

\AtBeginEnvironment{proof}{\setcounter{case}{0}}
{\setlength{\leftmargini}{1.5\parindent}
	\begin{itemize}
		\setlength{\itemsep}{-1.1mm}}
	{\end{itemize}}

\begin{document}
	\title{\bf Cycles and paths through vertices whose degrees are at least the bipartite-hole-number}
	\author[1]{ \bf Chengli Li\footnote{Email: lichengli0130@126.com.}}
	\author[2]{\bf  Feng Liu\footnote{Email: liufeng0609@126.com(corresponding author).}}
    \author[1]{\bf Yurui Tang\footnote{Email: tyr2290@163.com.}}
	
	\affil[1]{ \footnotesize Department of Mathematics,
		East China Normal University, Shanghai, 200241, China}
        \affil[2]{\footnotesize School of Mathematics Sciences, Shanghai Jiao Tong University,  Shanghai, 200240, China}
	\date{}
	\maketitle
	\begin{abstract}
The bipartite-hole-number of a graph $G$, denoted by $\widetilde{\alpha}(G)$, is the minimum integer $k$ such that there exist positive integers $s$ and $t$ with $s + t = k + 1$, satisfying the property that for any two disjoint sets $A, B \subseteq V(G)$ with $|A| = s$ and $|B| = t$, there is at least one edge between $A$ and $B$. In 1992, Bollobás and Brightwell, and independently Shi, proved that every $2$-connected graph of order $n$ contains a cycle passing through all vertices whose degrees are at least $\frac{n}{2}$. Motivated by their result, we show that in any $2$-connected graph of order $n$, there exists a cycle containing all vertices whose degrees are at least $\widetilde{\alpha}(G)$. Moreover, we prove that for any pair of vertices in a connected graph $G$, if their degrees are at least $\widetilde{\alpha}(G) + 1$, then there exists a path joining them that contains all vertices whose degrees are at least $\widetilde{\alpha}(G) + 1$. The results extend two existing ones.
			
		
\end{abstract}
	

\section{Introduction}	

We consider finite simple  graphs, and use standard terminology and notations from \cite{Bondy, West} throughout this article. We denote by $V(G)$ and $E(G)$ the vertex set and edge set of a graph $G,$ respectively, and denote by $|G|$ and $e(G)$ the order and size of $G,$ respectively. The neighborhood and degree of a vertex $x$ in a
 graph $G$ is denoted by $N_G(x)$ and $d_G(x),$ respectively. Let $N_G[x]=N_G(x)\cup \{x\}$, which is known as the closed neighborhood of $x$ in $G$. We denote by $\delta(G)$ the minimum degree of $G.$ For a vertex $x$ in $G$ and a vertex subset $S$ of $G$, $N_S(x)$ denotes  the set of neighbors of $x$ that are contained in $S$, and $N_S[x]=N_S(x)\cup \{x\}$. For convenience, we write $N_H(x)$ ($N_H[x]$, respectively) for $N_{V(H)}(x)$ ($N_{V(H)}[x]$, respectively) if $H$ is a subgraph of $G$.  
 We use $G[S]$ to denote the subgraph of $G$ induced by $S$, and let $G-S=G[V(G)\setminus S]$. 
 
 Given two vertex subsets $S$ and $T$ of $G$, we denote by $[S,T]$ the set of edges having one endpoint in $S$ and the other in $T$ of $G$.
 An $(S,T)$-path is a path which starts at a
 vertex of $S,$ ends at a vertex of $T,$ and whose internal vertices belong to neither $S$ nor $T$.
 In particular, if $S = \{x\}$, we write an $(\{x\},T)$-path as an $(x,T)$-path. An $(x,y)$-path is
 a path with endpoints $x$ and $y.$
The distance of two vertices $u$ and $v$ in a graph $G$ is denoted by $d_G(u,v).$
When there is no confusion, subscripts will be omitted. For $u,v\in V(G)$, we simply write $u\sim v$ if $uv \in E(G)$, and write $u\nsim v$ if $uv \notin E(G)$.

For a positive integer $k,$ the symbol $[k]$
 used in this article represents the set $\{1,2,\dots,k\}.$ Furthermore, for integers $a$ and $b$ with $a\le b,$ we
 use $[a, b]$ to denote the set of those integers $c$ satisfying $a\le c\le b.$
 
A Hamilton path in $G$ is a path containing every vertex of $G$.  A Hamilton  cycle in $G$ is a cycle containing every vertex of $G$. A graph $G$ is hamiltonian if it contains a Hamilton cycle. It is known that the problem of deciding whether a given graph is hamiltonian is NP-complete. Moreover, there exists no easily verifiable necessary and sufficient condition for the existence of a Hamilton cycle in a graph. Therefore it is natural and very interesting to study sufficient conditions for hamiltonicity. A classic result is Dirac's theorem from 1952, stated as follows.
\begin{theorem}[Dirac \cite{Dirac}]
Let $G$ be a graph of order at least three. If $\delta(G)\geq \frac{n}{2}$, then $G$ is hamiltonian.
\end{theorem}
Dirac's result initiated a new approach to establishing degree-based sufficient conditions for a graph to be hamiltonian. A lot of effort has been made by various people in generalization of Dirac’s theorem and this area is one of the core subjects in hamiltonian graph theory. For more information on some of these generalizations, we refer the reader to  \cite{Dirac,Fan1984,Faudree1989,Gould2014,Li2013,Ore1960}.
Based on Dirac's theorem, Bollobás and Brightwell, as well as Shi independently, discovered the following interesting theorem.

\begin{theorem}[Bollob\'{a}s-Brightwell \cite{Bollobas}, Shi \cite{Shi}]\label{th1}
In a $2$-connected graph of order $n$, there exists a cycle containing all vertices of degree at least $\frac{n}{2}$.
\end{theorem}

An interesting sufficient condition for hamiltonicity was given by McDiarmid and Yolov \cite{Mcdiarmid2017}. To state their result, we need the following notation. An $(s,t)$-bipartite-hole in a graph $G$ consists of two disjoint sets of vertices, $S$ and $T$, with $|S| = s$ and $|T| = t$, such that $[S, T] = \emptyset$. The {\it bipartite-hole-number} of a graph $G$, denoted by $\widetilde{\alpha}(G)$, is the minimum number $k$ such that there exist positive integers $s$ and $t$ with $s + t = k + 1$, and such that $G$ does not contain an $(s,t)$-bipartite-hole. An equivalent definition of $\widetilde{\alpha}(G)$ is the maximum integer $r$ such that $G$ contains an $(s,t)$-bipartite-hole for every pair of nonnegative integers $s$ and $t$ with $s+t=r$. The  notation of bipartite-hole-number was introduced by 
McDiarmid and Yolov \cite{Mcdiarmid2017} in the study of Hamilton cycles. Now, let us state the result of McDiarmid and Yolov.
\begin{theorem}[McDiarmid-Yolov \cite{Mcdiarmid2017}]\label{dirac-hamiltonian}
Let $G$ be a graph of order at least three.  If $\delta(G)\ge \widetilde{\alpha}(G)$, then $G$ is hamiltonian.
\end{theorem}
This result generalizes Dirac's theorem. Specifically, if  $\delta(G) \geq \frac{n}{2}$, then we have $\lceil \frac{n}{2} \rceil \geq \widetilde{\alpha}(G)$. This holds because for any sets $A, B \subseteq V(G)$ with $|A|=1$ and $|B|=\lceil \frac{n}{2} \rceil$, there must be at least one edge connecting $A$ and $B$. Consequently, the inequality $\delta(G) \geq \lceil \frac{n}{2} \rceil \geq \widetilde{\alpha}(G)$ follows, and by Theorem~\ref{dirac-hamiltonian}, the graph $G$ is hamiltonian.  Motivated by Theorem \ref{th1} and Theorem \ref{dirac-hamiltonian}, we establish the following result.
\begin{theorem}\label{Theorem-dirac-vertex-set}
Let $G$ be a $2$-connected graph. Then $G$ contains a cycle passing through all vertices of degree at least $\widetilde{\alpha}(G)$.
\end{theorem}
Our result extends Theorem~\ref{dirac-hamiltonian} and thus generalizes Dirac's theorem. Indeed, if $\delta(G) \geq \widetilde{\alpha}(G)$, then by Theorem~\ref{Theorem-dirac-vertex-set}, $G$ contains a cycle passing through all its vertices. Consequently, $G$ is hamiltonian.

A graph is called hamiltonian-connected if between any two distinct vertices there is a Hamilton path. It is obvious that every hamiltonian-connected graph of order at least three is hamiltonian, but the converse is not true (see, for example, the bipartite graph). The following well-known theorem, established by Ore, provides the corresponding degree sum conditions for any graph to be hamiltonian-connected.
\begin{theorem}[Ore \cite{Ore1963}]\label{Theorem-ore-hamiltonian-connected}
Let $G$ be a graph of order at least three. If $d_G(x)+d_G(y)\geq n+1$ for every pair of nonadjacent vertices $x$ and $y$, then $G$ is hamiltonian-connected.
\end{theorem}
As a corollary, Theorem~\ref{Theorem-ore-hamiltonian-connected} implies a result of Erd\H{o}s and Gallai \cite{Erdos1959}.
\begin{theorem}[Erd\H{o}s-Gallai  \cite{Erdos1959}]
Let $G$ be a graph of order at least three.  If $\delta(G)\geq \frac{n+1}{2}$, then $G$ is hamiltonian-connected.
\end{theorem}

In 2024, Zhou, Broersma, Wang and Lu provided a sufficient condition for hamiltonian connectedness based on the minimum degree and the bipartite-hole-number.
\begin{theorem}[Zhou-Broersma-Wang-Lu \cite{Zhou2024}]\label{Theorem-Zhou}
Let $G$ be a graph of order at least three. If $\delta(G)\ge
\widetilde{\alpha}(G)+1,$ then $G$ is hamiltonian-connected.
\end{theorem}

Inspired by Theorem~\ref{Theorem-Zhou}, we present our second result as follows.

\begin{theorem}\label{Theorem-exist-path}
Let $G$ be a graph and let $u,v$ be two vertices of degree at least $\widetilde{\alpha}(G)+1.$
Then there exists a $(u,v)$-path of $G$ which contains all vertices of degree at least $\widetilde{\alpha}(G)+1.$
\end{theorem}
Our result extends Theorem~\ref{Theorem-Zhou}. Indeed, if $\delta(G) \ge \widetilde{\alpha}(G) + 1$, then by Theorem~\ref{Theorem-exist-path}, there exists a path between any two distinct vertices that passes through all the vertices of $G$. Consequently, $G$ is hamiltonian-connected.

We organize the remainder of this paper as follows:  Section \ref{Proof-hamiltonian-traceable} presents the proofs of Theorems \ref{Theorem-dirac-vertex-set} and \ref{Theorem-exist-path}, while Section~\ref{Proof-corollary} presents  an application of Theorem~\ref{Theorem-dirac-vertex-set}.

\section{Proofs of Theorems \ref{Theorem-dirac-vertex-set} and \ref{Theorem-exist-path}}\label{Proof-hamiltonian-traceable}
Before proceeding with the proof, we list  some notations that will be needed in later proofs.
Let $P$ be an oriented $(u,v)$-path. For $x\in V(P)$ with $x\ne v$, denote by $x^+$ the immediate successor on $P$.  For $x\in V(P)$  with $x\ne u$, denote by $x^-$ the predecessor on $P$. 
For $S\subseteq V(P)$, let $S^+=\{x^+:x\in S\setminus\{v\}\}$ and $S^-=\{x^-:x\in S\setminus\{u\}\}$. Obviously, $|S^+|=|S|$ or $|S^+|=|S|-1$. 
For $x,y\in V(P)$,  $\overrightarrow{P}[x,y]$ denotes the segment of $P$ from $x$ to $y$ which follows the orientation of $P$, while $\overleftarrow{P}[x,y]$ denotes the opposite segment of $P$ from $y$ to $x$. Particularly, if $x=y$, then $\overrightarrow{P}[x,y]=\overleftarrow{P}[x,y]=x$. Denote by $P^*$ the set of internal vertices of $P.$

\begin{proof}[\bf Proof of Theorem \ref{Theorem-dirac-vertex-set}]
We prove Theorem \ref{Theorem-dirac-vertex-set} by contradiction. Let $G$ be a counterexample to Theorem \ref{Theorem-dirac-vertex-set} of minimum order, subject to the first condition, $G$ has maximum size.
 Observe that there exist two nonadjacent vertices, each of degree at least $\widetilde{\alpha}(G)$. Indeed, otherwise the set of all vertices of degree at least $\widetilde{\alpha}(G)$ forms a clique, which naturally contains a cycle passing through all such vertices. Then $G$ is non-complete, and hence $\widetilde{\alpha}(G)\geq 2$. Let $s\in [t]$ satisfy $\widetilde{\alpha}(G)+1=s+t$, and  assume that $G$ has no  $(s,t)$-bipartite-hole. Without loss of generality, let $u, v \in V(G)$ be two nonadjacent vertices with $\min\{d_G(u), d_G(v)\} \geq \widetilde{\alpha}(G)$.

Now, denote by $G_{uv}$ the graph obtained from $G$ by adding a new edge $uv$. Note that adding edges does not increase the bipartite-hole-number. Therefore, by the choice of $G$, we have that $G_{uv}$ contains a cycle passing through all vertices of degree at least $\widetilde{\alpha}(G_{uv})$.  
This implies that there is a $(u,v)$-path  $P$ passing through all vertices of degree at least $\widetilde{\alpha}(G)$ in $G$. For convenience, assume that $P=v_1v_2\ldots v_k$ with $v_1=u$ and $v_k=v$.
Denote $U_1=N_G(u)\setminus V(P)$ and $V_1=N_G(v)\setminus V(P)$.  Note that $U_1\cap V_1=\emptyset$. From now on, we will differentiate between two scenarios.

\begin{claim}\label{Claim-U1}
$[U_1,\,N_P(v)^+\cup V_1]= \emptyset$.
\end{claim}
Otherwise, let $x \in U_1$ and $y \in N_P(v)^+ \cup V_1$ such that $x \sim y$.
\begin{itemize}
    \item If $y \in V_1$, then $v_1\overrightarrow{P}[v_1,v_k]v_k y x v_1$ forms a cycle passing through all vertices of degree at least $\widetilde{\alpha}(G)$, a contradiction. 
    \item If $y\in N_P(v)^+$, then $v_1\overrightarrow{P}[v_1,y^-]y^- v_k \overleftarrow{P}[v_k,y]yx v_1$ forms a cycle passing through all vertices of degree at least $\widetilde{\alpha}(G)$, a contradiction. 
\end{itemize}
Therefore, $[U_1,\,N_P(v)^+\cup V_1]= \emptyset$. This proves Claim~\ref{Claim-U1}.

\begin{claim}\label{Claim-U1-size}
    $|U_1|\leq s-1$.
\end{claim}
Otherwise, suppose that $|U_1|\geq s$. By Claim~\ref{Claim-U1}, we have that $[U_1,\,N_P(v)^+\cup V_1]= \emptyset$.  Since $G$ contains no $(s,t)$-bipartite-hole, $|N_P(v)^+\cup V_1|\leq t-1$. This implies that 
\begin{flalign*}
\widetilde{\alpha}(G)\leq d_G(v)=|N_P(v)^+\cup V_1|\leq t-1<\widetilde{\alpha}(G),
\end{flalign*}
a contradiction. This proves Claim~\ref{Claim-U1-size}.



By Claim~\ref{Claim-U1-size}, we have that $|U_1|\leq s-1$. Now, we consider a partition of the neighborhoods of $v_1$ and $v_k$. Since  $1\leq s <\widetilde{\alpha}(G)\leq d_G(v_1)$, there exists an integer $r\in [2,\,k-1]$ such that
\begin{flalign*}
|N_G(v_1)\cap \{v_i:~i\in [2,\,r]\} |=s-|U_1|.    
\end{flalign*}
Denote that $U_2=N_G(v_1)\cap \{v_i:~i\in [2,\,r]\}$ and $U_3=N_G(v_1)\cap \{v_i:~i\in [r+1,\,k-1]\}$. That is, $N_G(u)=U_1\cup U_2\cup U_3$.  Next, denote $V_2=N_G(v_n)\cap\{v_j:~j\in [r,\,k-1]\}$ and $V_3=N_G(v_n)\cap\{v_j:~j\in [2,\,r-1]\}$. That is, $N_G(v)=V_1\cup V_2\cup V_3$.

Suppose that $[U_2^-, V_1 \cup V_2^+]\neq \emptyset$. Let $x \in  U_2^-$ and $y \in V_1 \cup V_2^+$ such that $x \sim y$. 
\begin{itemize}
    \item If  $y\in V_1$, then $v_1\overrightarrow{P}[v_1,x]xyv_k\overrightarrow{P}[v_k,x^+]x^+v_1$ forms a cycle passing through all vertices of degree at least $\widetilde{\alpha}(G)$, a contradiction. 
    \item If  $y\in V_2^+$, then $v_1\overrightarrow{P}[v_1,x]xy\overrightarrow{P}[y,v_k]v_ky^-\overleftarrow{P}[y^-,x^+]x^+v_1$ forms a cycle passing through all vertices of degree at least $\widetilde{\alpha}(G)$, a contradiction. 
\end{itemize}
Therefore, $[U_2^-, V_1 \cup V_2^+]=\emptyset$. By Claim~\ref{Claim-U1}, we have that $[U_1 \cup U_2^-, V_1 \cup V_2^+] = \emptyset$.
Since $G$ has no $(s,t)$-bipartite-hole, $|V_1\cup V_2|\leq t-1$. Note that $d_G(v)\geq \widetilde{\alpha}(G)$. This implies that 
\[
|V_3|\geq d_G(v)-(t-1)\geq \widetilde{\alpha}(G)-(t-1)=s.
\]
Suppose that $[V_3^+, U_3^+\cup \{v_1\}]\neq \emptyset$. Let $x\in V_3^+ $ and $y \in U_3^+\cup \{v_1\}$  such that $x \sim y$.  
\begin{itemize}
    \item If $y=v_1$, then $v_1\overrightarrow{P}[v_1,x^-]x^-v_k\overleftarrow{P}[v_k,x]xv_1$ forms a cycle passing through all vertices of degree at least $\widetilde{\alpha}(G)$, a contradiction. 
    \item If $y\in U_3^+$,  then $v_1\overrightarrow{P}[v_1,x^-]x^-v_k\overleftarrow{P}[v_k,y]yx\overrightarrow{P}[x,y^-]y^-v_1$ forms a cycle passing through all vertices of degree at least $\widetilde{\alpha}(G)$, a contradiction.
\end{itemize}
Therefore, $[V_3^+, U_3^+\cup \{v_1\}]=\emptyset$. 
Since $G$ has no $(s,t)$-bipartite-hole, $|U_3\cup \{v_1\}|\leq t-1$. That is, $|U_3|\leq t-2$. However, it follows that 
\begin{flalign*}
\widetilde{\alpha}(G)\leq d_G(v_1)=|U_1|+|U_2|+|U_3|\leq s+t-2=\widetilde{\alpha}(G)-1,
\end{flalign*}
a contradiction. This completes the proof of Theorem \ref{Theorem-dirac-vertex-set}.
\end{proof}

\begin{proof}[\bf Proof of Theorem \ref{Theorem-exist-path}]
\setcounter{claim}{0}
We prove Theorem~\ref{Theorem-exist-path} by contradiction. Let $G$ be a counterexample to Theorem~\ref{Theorem-exist-path}.
Thus there exist two vertices $u,v$ in $G$ such that every $(u,v)$-path of $G$ does not contain all vertices of degree at least $\widetilde{\alpha}(G)+1.$ We assume that $P$ is a $(u,v)$-path of $G$ containing as many vertices of degree at least $\widetilde{\alpha}(G)+1$ as possible.
Let $R = V(G) \setminus V(P)$, and let $w \in R$ such that $d_G(w) \geq \widetilde{\alpha}(G) + 1$.

Since $G$ is connected, there exists a $(w,V(P))$-path  $Q$.  Without loss of generality,  we may choose the paths $P$ and $Q$ in such a way that the length of $Q$ is as small as possible. For convenience, assume that $P=v_1v_2\dots v_k$ with $v_1=u$ and $v_k=v.$ Let $v_p$ be the other endpoint of $Q$ distinct from $w,$ and assume that $Q$ is a path from $v_p$ to $w.$ Furthermore, suppose to the symmetry that $p<k$. Let $q\geq p+1$ such that $v_q$ has degree at least $\widetilde{\alpha}(G)+1$ and each vertex $v_i$ with $p<i<q$ has degree at most $\widetilde{\alpha}(G)$. 

\begin{claim}\label{Claim-R}
$N_R(w)\cap N_R(v_q)=\emptyset$
\end{claim} 
Indeed, by the choices of $P$ and $Q$, we have $N_{Q^*}(v_q)=\emptyset$ and so $N_{Q^*}(w)\cap N_{Q^*}(v_q)=\emptyset$. If $(N_R(w)\cap N_R(v_q))\setminus V(Q)\neq \emptyset,$ let $x\in (N_R(w)\cap N_R(v_q))\setminus V(Q),$ then
\[
u\overrightarrow{P}[u,v_p]v_p\overrightarrow{Q}[v_p,w]wxv_q\overrightarrow{P}[v_q,v]v
\]
is a $(u,v)$-path containing more vertices of degree at least $\widetilde{\alpha}(G)+1,$ a contradiction. This proves Claim~\ref{Claim-R}.

Let $s\in [t]$ satisfy $\widetilde{\alpha}(G)+1=s+t$, and  assume that $G$ has no  $(s,t)$-bipartite-hole.

\begin{claim}\label{Claim-R-neighbor-size}
$|N_R(w)|\leq t-1$.
\end{claim}

Otherwise, suppose $|N_R(w)|\geq t$. By Claim~\ref{Claim-R}, we have that $N_R(w)\cap N_R(v_q)=\emptyset$. Note that $t\ge 2,$ otherwise we have $s=t=1,$ and so $\widetilde{\alpha}(G)=1.$ This implies that $G$ is a complete graph, a contradiction. Now,
\begin{flalign*}
|\big(N_P(v_q)^-\cup N_R(v_q)\big)\setminus \{v_p\}|\ge d_G(v_q)-2\ge s+t-2\ge s.
\end{flalign*}
Since  $G$ has no $(s,t)$-bipartite-hole, 
\begin{flalign*}
[N_R[w]\setminus V(Q),\,\big(N_P(v_q)^-\cup N_R(v_q)\big)\setminus \{v_p\}]\neq \emptyset.
\end{flalign*}
Let $x\in N_R[w]\setminus V(Q)$ and $y\in (N_P(v_q)^-\cup N_R(v_q))\setminus \{v_p\}$  such that $x\sim y$.  
\begin{itemize}
    \item If $y\in N_R(v_q)$, then $u\overrightarrow{P}[u,v_p]v_p\overrightarrow{Q}[v_p,w]wxyv_q\overrightarrow{P}[v_q,v]$ is a $(u,v)$-path containing more vertices of degree at least $\widetilde{\alpha}(G)+1,$ a contradiction.
    \item If $y\in N_P(v_q)^-\setminus \{v_p\}$, then there are three possible $(u,v)$-paths:
    \begin{flalign*}
        \begin{cases}
            u\overrightarrow{P}[u,y]yxw\overleftarrow{Q}[w,v_p]v_p\overleftarrow{P}[v_p,y^+]y^+v_q\overrightarrow{P}[v_q,v]v,&\text{if $y\in V(P[u,v_{p-1}])$};\\
            u\overrightarrow{P}[u,v_p]v_p\overrightarrow{Q}[v_p,w]wxy\overleftarrow{P}[y,v_q]v_qy^+\overrightarrow{P}[y^+,v]v,&\text{if $y\in V(P[v_q,v])$};\\
           u\overrightarrow{P}[u,v_p]v_p\overrightarrow{Q}[v_p,w]wxy\overrightarrow{P}[y,v], &\text{if $y\in V(P[v_{p+1},v_q])$}.
        \end{cases}
    \end{flalign*}
    However, regardless of the case, it follows that $G$ has a $(u,v)$-path containing more vertices of degree at least $\widetilde{\alpha}(G)+1$, a contradiction.
\end{itemize}
Therefore, $|N_R(w)|\leq t-1$. This proves Claim~\ref{Claim-R-neighbor-size}.

\begin{claim}\label{Claim-Q}
   $|N_P(w)|\geq s+1$. 
\end{claim}
By Claim~\ref{Claim-R-neighbor-size}, we have that $|N_R(w)|\leq t-1$. Since $d_G(w)\geq \widetilde{\alpha}(G)+1$, 
\begin{flalign*}
|N_P(w)|=d_G(w)-|N_R(w)|\geq  \widetilde{\alpha}(G)+1-(t-1)= s+1.
\end{flalign*} 
This proves Claim~\ref{Claim-Q}.

By Claim~\ref{Claim-Q}, there exists an integer $r\in [k]$ such that
\begin{flalign*}
|N_G(w)\cap \{v_i:~i\in [r]\} |=s+1.    
\end{flalign*}
Moreover, we choose the integer $r$ as small as possible.
This implies that $w\sim v_r$ and $w\nsim v_{r-1}$. From now on, we will differentiate between two cases.
\begin{case}
$r<k$.
\end{case}
Without loss of generality, we may assume that $Q$ is chosen as $wv_r$.
In this case, the integer $q$ exists.
Let $W_2=N_G(w)\cap \{v_i:~i\in [r-1]\}.$
We now divide $N_P(v_q)$ into the following three subsets:
\begin{flalign*}
\begin{cases}
V_{1}=N_G(v_q)\cap\{v_j:~j\in [r+1,\,q]\};\\
V_{2}=N_G(v_q)\cap\{v_j:~j\in [q+1,\,k]\};\\
V_3=N_G(v_q)\cap\{v_j:~j\in [r]\}.
\end{cases}
\end{flalign*}
Suppose that $[W_2^+, N_R(v_q) \cup V_{1} \cup V_{2}^-] \neq \emptyset$. Let $x \in W_2^+$ and $y \in N_R(v_q) \cup V_{1} \cup V_{2}^-$ be such that $x \sim y$. 
\begin{itemize}
    \item  If $y \in N_R(v_q)$, then  $u\overrightarrow{P}[u,x^-]x^-wv_r\overleftarrow{P}[v_r,x]xyv_q\overrightarrow{P}[v_q,v]v$ is a $(u,v)$-path containing more vertices of degree at least $\widetilde{\alpha}(G)+1,$ a contradiction.
    \item If $y \in V_1$, then $u\overrightarrow{P}[u,x^-]x^-wv_r\overleftarrow{P}[v_r,x]xy\overrightarrow{P}[y,v]v$ is a $(u,v)$-path containing more vertices of degree at least $\widetilde{\alpha}(G)+1,$ a contradiction.
    \item If $y\in V_{2}^-$, then $u\overrightarrow{P}[u,x^-]x^-wv_r\overleftarrow{P}[v_r,x]xy\overleftarrow{P}[y,v_q]v_qy^+\overrightarrow{P}[y^+,v]v$ is a $(u,v)$-path containing more vertices of degree at least $\widetilde{\alpha}(G)+1,$ a contradiction.
\end{itemize}
 Therefore, $[W_2^+, N_R(v_q) \cup V_{1} \cup V_{2}^-] =\emptyset$. Since $G$ has no $(s,t)$-bipartite-hole, $|N_R(v_q) \cup V_{1} \cup V_{2}^-|\leq t-1$. Recall that $d_G(v_q)\geq \widetilde{\alpha}(G)+1$. This implies that 
\[
|V_3|\geq d_G(v_q)-(t-1)\geq \widetilde{\alpha}(G)+1-(t-1)=s+1.
\]
Therefore, $|V_3^-|\geq s$. Let $W_3=N_G(w)\cap \{v_i:i\in [r+1,k]\}$. In fact, $W_3=N_G(w)\cap \{v_i:i\in [q+1,k]\}$. Indeed, if $w$ has a neighbor in $P[v_{r+1}, v_q]$, say $x$, then $u\overrightarrow{P}[u, v_r]v_rwx\overrightarrow{P}[x, v]v$ is a $(u, v)$-path containing more vertices of degree at least $\widetilde{\alpha}(G)+1$. This contradicts the choice of $P$.

Suppose that $[V_3^-,N_R[w]\cup W_3^-]\neq \emptyset$.
Let $x\in V_3^-$ and $y\in N_R[w]\cup W_3^-$ be such that $x\sim y$. 
\begin{itemize}
\item If $y\in N_R[w]$, then $u\overrightarrow{P}[u,x]xywv_r\overleftarrow{P}[v_r,x^+]x^+v_q\overrightarrow{P}[v_q,v]v$ is a $(u,v)$-path containing more vertices of degree at least $\widetilde{\alpha}(G)+1,$ a contradiction.

\item If $y\in W_3^-$, then $u\overrightarrow{P}[u,x]xy\overrightarrow{P}[y,v_q]v_qx^+\overrightarrow{P}[x^+,v_r]v_rwy^+\overrightarrow{P}[y^+,v]v$ is a $(u,v)$-path containing more vertices of degree at least $\widetilde{\alpha}(G)+1,$ a contradiction.
    
\end{itemize}
Therefore, $[V_3^-,N_R[w]\cup W_3^-]=\emptyset$. Since $G$ has no $(s,t)$-bipartite-hole, $|N_R[w]\cup  W_3^-|\leq t-1$. This implies that 
\[
d_G(w)=|N_R(w)|+|W_3|+|W_2|+1\le (t-2)+s+1=s+t-1=\widetilde{\alpha}(G),
\]
which contradicts with $d_G(w)\ge \widetilde{\alpha}(G)+1.$
\begin{case}
$r=k$.
\end{case}
In this case, by Claim~\ref{Claim-R-neighbor-size}, we have $|N_R(w)|=t-1$.
We choose an integer $h \in [k-1]$ such that $v_h$ has degree at least $\widetilde{\alpha}(G)+1$, and $h$ is as large as possible.
By the previous analysis of Claim~\ref{Claim-U1}, we know that
 $N_R(v_{h})\cap N_R[w]=\emptyset$. In fact, $[N_R[w],N_R(v_h)]=\emptyset.$  Otherwise, $G$ contains a $(v_h, v_k)$-path $Q'$ that contains $w$ and is internally disjoint from $P$. However, it follows that $u\overrightarrow{P}[u,v_h]v_h\overrightarrow{Q'}[v_h,v_k]v_k$ is a $(u,v)$-path containing more vertices of degree at least $\widetilde{\alpha}(G)+1,$ a contradiction. Since $G$ has no $(s,t)$-bipartite-hole, $|N_R(v_h)|\leq s-1$. This implies that  
 \begin{flalign*}
    |N_P(v_h)|=d_G(v_h)-|N_R(v_h)|\geq \widetilde{\alpha}(G)+1-(s-1)= t+1.
 \end{flalign*}
 We now divide $N_P(v_h)$ into the following two subsets:
\begin{flalign*}
\begin{cases}
A_1=N_G(v_{h})\cap \{v_i:i\in [h]\};\\
A_2=N_G(v_{h})\cap \{v_i:i\in [h+1,k-1]\}.
\end{cases}
\end{flalign*}
Then $|A_1^+\cup A_2|\geq t\ge s$. Since $G$ has no $(s,t)$-bipartite-hole, $[N_R[w], A_1^+\cup A_2]\neq \emptyset.$ Let $x\in N_R[w]$ and $y\in A_1^+\cup A_2$ be such that $x\sim y$.  
\begin{itemize}
    \item If $y\in A_1^+$,then $u\overrightarrow{P}[u,y^-]y^-v_h\overleftarrow{P}[v_h,y]yxwv_k$  is  a $(u,v)$-path containing more vertices of degree at least $\widetilde{\alpha}(G)+1$, a contradiction.
    \item If $y\in A_2$, then $u\overrightarrow{P}[u,y]yxw  v_k$ is  a $(u,v)$-path containing more vertices of degree at least $\widetilde{\alpha}(G)+1$, a contradiction.
\end{itemize}

This completes the proof of Theorem~\ref{Theorem-exist-path}.
\end{proof}

\section{Applications of Theorem~\ref{Theorem-dirac-vertex-set}}\label{Proof-corollary}
Let $G$ be a graph. For a vertex $u\in V(G)$ and a nonempty set $S\subseteq V(G)$, let $d(u,S)=\min\{d(u,v):v\in S\}$, and let $N_i(S)=\{u\in V(G):d(u,S)=i\}$ for any integer $i\geq 0$.
Denote by $\alpha(G)$ the independence number of $G$.
Let $x,y\in V(G)$ with $d(x,y)=2$. We define $I(x,y)=|N(x)\cap N(y)|$ and 
$$\alpha_2(x,y)=\left\{\begin{matrix}
  \alpha (G[N_2(x)\cap N_2(y)])& \,\textup{if}\,\,\, N_2(x)\cap N_2(y)\neq \emptyset;\\
 0 & \textup{otherwise}.
\end{matrix}\right. $$

\begin{theorem}\label{Theorem-fan}

Let $G$ be a $2$-connected graph. If $I(x,y)\geq \alpha_2(x,y)+2$ whenever $d(x,y)=2$ and $\max\{d_G(x),\,d_G(y)\}<\tilde{\alpha}(G)$, then $G$ is hamiltonian.
\end{theorem}	
\begin{proof}[\bf Proof of Theorem~\ref{Theorem-fan}]
\setcounter{claim}{0}
By Theorem \ref{Theorem-dirac-vertex-set}, there exists a cycle containing all vertices of degree at least $\widetilde{\alpha}(G)$. Let $C=v_1v_2\dots v_kv_1$ be a cycle with maximum length that contains all vertices of degree at least $\widetilde{\alpha}(G)$. We assume that $R=V(G)\setminus V(C)$. If $R=\emptyset$, we completes the proof. Next suppose $R\neq \emptyset.$

Since $G$ is $2$-connected, there exists a path $Q$ of length at least $2$ connecting two vertices of $C$ that is internally vertex-disjoint from $C$. Without loss of generality, suppose that $V(Q)\cap V(C)=\{v_1,v_t\}$.  Let $x$ be the successor of $v_1$ and $x'$ the predecessor of $v_t$ on $Q,$ where possibly $x=x'$. Since $C$ is a cycle with maximum length that contains all vertices of degree at least $\widetilde{\alpha}(G)$, then $d_G(x)< \widetilde{\alpha}(G)$, $d_G(x')< \widetilde{\alpha}(G)$ and $3\le t\le k-1$.

\begin{claim}\label{Claim-degree}
        $\min\{d_G(v_2),d_G(v_{t+1})\}<\widetilde{\alpha}(G)$.
\end{claim} 
 To the contrary, assume that $\min\{d_G(v_2),d_G(v_{t+1})\}\ge\widetilde{\alpha}(G)$. Consider the path $$P=v_2v_3\dots v_t Q v_1 v_kv_{k-1}\dots v_{t+1}.$$ 
For convenience, we relabel the vertices of $P$  as $P=u_1u_2\dots u_m,$ where $u_1=v_2$ and $u_m=v_{t+1}$. 
 Clearly, $V(C)\subsetneq V(P).$ Denote $X_1=N_G(u_1)\setminus V(P)$ and $Y_1=N_G(u_m)\setminus V(P).$ Since $G$ is non-complete, we have  $\widetilde{\alpha}(G)\geq 2$. Let $s\in [t]$ satisfy $\widetilde{\alpha}(G)+1=s+t$, and  assume that $G$ has no $(s,t)$-bipartite-hole.  

Suppose that $[X_1,N_P^+(u_m)\cup Y_1]\neq \emptyset$. Let $z\in X_1$ and $w\in N_P^+(u_m)\cup Y_1$ be such that $z\sim w$.
\begin{itemize}
    \item  If $w\in N_P^+(u_m)$, then $u_1\overrightarrow{P}[u_1,w^-]w^-u_m\overleftarrow{P}[u_m,w]wzu_1$ is a cycle containing all the vertices on $P$, a contradiction.
    \item If $w\in Y_1$, then $u_1\overrightarrow{P}[u_1,u_m]wzu_1$ is a cycle containing all the vertices on $P$, a contradiction.
\end{itemize}
Therefore, we must have $[X_1, N_P^+(u_m) \cup Y_1] =\emptyset$.
Since $d_G(u_m) \geq \widetilde{\alpha}(G)$, it follows that $|N_P^+(u_m) \cup Y_1| \geq t$.
Moreover, as $G$ contains no $(s,t)$-bipartite-hole and $[X_1, N_P^+(u_m) \cup Y_1] = \emptyset$, we obtain $|X_1| \leq s - 1$.
Note that $1 \leq s < \widetilde{\alpha}(G) \le d_G(u_1)$.
Thus, there exists an integer $r \in [2,,k-1]$ such that
\begin{flalign*}
|N_G(u_1) \cap \{u_i :~ i \in [2,,r]\}| = s - |X_1|.
\end{flalign*}
To facilitate our analysis, we partition the neighborhood of $u_1$ on $P$ as follows:
\begin{flalign*}
\begin{cases}
X_2 = N_G(u_1) \cap \{u_i :~ i \in [2,r]\},\\
X_3 = N_G(u_1) \cap \{u_i :~ i \in [r+1,m-1]\}.
\end{cases}
\end{flalign*}
Similarly, we divide the neighborhood of $u_m$ on $P$ into the sets:
\begin{flalign*}
\begin{cases}
Y_2 = N_G(u_m) \cap \{v_j :~ j \in [r,m-1]\},\\
Y_3 = N_G(u_m) \cap \{v_j :~ j \in [2,r-1]\}.
\end{cases}
\end{flalign*}

Suppose that $[X_2^-, Y_1 \cup Y_2^+]\neq \emptyset$.  Let $z \in  X_2^-$ and $w \in Y_1 \cup Y_2^+$ be such that $z \sim w$. 
\begin{itemize}
    \item If $w\in Y_1$, then $u_1\overrightarrow{P}[u_1,z]zwu_m\overleftarrow{P}[u_m,z^+]z^+u_1$ is a cycle containing all the vertices on $P$, a contradiction.
    \item If $w\in Y_2^+$, then $u_1\overrightarrow{P}[u_1,z]zw\overrightarrow{P}[w,u_m]u_mw^-\overleftarrow{P}[w^-,z^+]z^+u_1$ is a cycle containing all the vertices on $P$, a contradiction.
\end{itemize}
Therefore, $[X_2^-, Y_1 \cup Y_2^+]= \emptyset$. By the previous analysis, we have that 
\begin{flalign*}
[X_1 \cup X_2^-, Y_1 \cup Y_2^+]= \emptyset.
\end{flalign*}

Since $G$ has no $(s,t)$-bipartite-hole, $|Y_1\cup Y_2|\leq t-1$. Note that $d_G(u_m)\geq \widetilde{\alpha}(G)$. This implies that 
\[|Y_3|\geq d_G(u_m)-(t-1)\geq \widetilde{\alpha}(G)-(t-1)=s.
\]

Suppose that $[Y_3^+, X_3^+\cup \{u_1\}]\neq \emptyset$. Let $z \in Y_3^+ $ and $w \in X_3^+\cup \{u_1\}$ be such that $z \sim w$.  
\begin{itemize}
    \item If $w\in X_3^+$, then $u_1\overrightarrow{P}[u_1,z^-]z^-u_m\overleftarrow{P}[u_m,w]wz\overrightarrow{P}[z,w^-]w^-u_1$ is a cycle containing all the vertices on $P$, a contradiction. 
    \item If $w=u_1$, then $u_1\overrightarrow{P}[u_1,z^-]z^-u_m\overleftarrow{P}[u_m,z]zu_1$ is a cycle containing all the vertices on $P$, a contradiction. 
\end{itemize}
Therefore, $[Y_3^+, X_3^+\cup \{u_1\}]=\emptyset$.
Since $G$ has no $(s,t)$-bipartite-hole, $|X_3\cup \{u_1\}|\leq t-1$. That is, $|X_3|\leq t-2$. However, it follows that 
\begin{flalign*}
\widetilde{\alpha}(G)\leq d_G(u_1)=|X_1|+|X_2|+|X_3|\leq s+t-2=\widetilde{\alpha}(G)-1,
\end{flalign*}
a contradiction. This proves Claim~\ref{Claim-degree}.

By Claim~\ref{Claim-degree}, we have $\min\{d_G(v_2),d_G(v_{t+1})\}<\widetilde{\alpha}(G)$. Without loss of generality, we assume that $d_G(v_2)<\widetilde{\alpha}(G)$. Note that $N_G(x) \cap N_G(v_2) \subseteq V(C)$. Otherwise, there would exist a cycle containing $x$ and all the vertices of $C$, a contradiction. Let $N_G(x)\cap N_G(v_2)=\{ v_1,v_{i_1},v_{i_2},\dots,v_{i_r}\}$, where $1<i_1<\dots <i_r<k.$ 

We assert that $\{v_{i_1+1},v_{i_2+1},\dots,v_{i_r+1}\}$ is an independent set. To the contrary,  we may assume that $v_{i_j+1}v_{i_{j'}+1}\in E(G)$ where $j<j'$. Then 
\begin{flalign*}
v_2v_3\dots v_{i_j}xv_1v_kv_{k-1}\dots v_{i_{j'}+1}v_{i_j+1}v_{i_j+2}\dots v_{i_{j'}}v_2
\end{flalign*} 
is a cycle containing the vertex $x$ and all the vertices on $C,$ a contradiction.
Clearly, $\big(N_G(x)\cup N_G(v_2)\big)\cap \{v_{i_1+1},v_{i_2+1},\dots,v_{i_r+1}\}=\emptyset$.  Therefore, $\{v_{i_1+1},v_{i_2+1},\dots,v_{i_r+1}\}\subseteq N_2(x)\cap N_2(v_2)$.
This implies that  $d(v_2,x)=2$ with $I(v_2,x)-1\le \alpha_2(v_2,x)$, a contradiction. 

This completes the proof of Theorem~\ref{Theorem-fan}.
\end{proof}

As a direct corollary of Theorem~\ref{Theorem-fan}, we have the following consequence, which was proved by Liu, Yuan, and Zhang~\cite{Liu2025}.
\begin{corollary}[Liu-Yuan-Zhang \cite{Liu2025}]\label{Corollary-fan}
    Let $G$ be a $2$-connected graph. If $\min\{\max\{d_G(x),d_G(y)\}:d(x,y)=2\}\ge \widetilde{\alpha}(G),$ then $G$ is hamiltonian.
\end{corollary}


	\section*{Acknowledgement}  The authors are grateful to Professor Xingzhi Zhan for his constant support and guidance. This research  was supported by the NSFC grant 12271170 and Science and Technology Commission of Shanghai Municipality (STCSM) grant 22DZ2229014.
	
	\section*{Declaration}
	
	\noindent$\textbf{Conflict~of~interest}$
	The authors declare that they have no known competing financial interests or personal relationships that could have appeared to influence the work reported in this paper.
	
	\noindent$\textbf{Data~availability}$
	Data sharing not applicable to this paper as no datasets were generated or analysed during the current study.

\end{document}